\documentclass[11pt,letterpaper]{amsart}
\usepackage{amsmath,amsthm,amsfonts,amscd,amssymb,eucal,latexsym,mathrsfs, stmaryrd,enumitem,bm,mathtools,microtype,aligned-overset}

\newtheorem{theorem}{Theorem}[section]

\newtheorem{lemma}[theorem]{Lemma}
\newtheorem{proposition}[theorem]{Proposition}

\theoremstyle{definition}

\newtheorem{example}[theorem]{Example}

\newcounter{theoremintro}
\newtheorem{theoremi}[theoremintro]{Theorem}
\newtheorem{corollaryi}[theoremintro]{Corollary}

\newcommand{\cZ}{{\mathcal Z}}

\newcommand{\Zb}{{\mathbb Z}}
\newcommand{\Nb}{{\mathbb N}}
\newcommand{\Rb}{{\mathbb R}}

\newcommand{\eps}{\varepsilon}

\newcommand{\sC}{{\mathscr C}}

\newcommand{\sZ}{{\mathscr Z}}
\newcommand{\erg}{{\rm erg}}

\numberwithin{equation}{section}

\DeclareMathOperator{\Sym}{Sym}

\allowdisplaybreaks

\begin{document}

\title{Bauer simplices and the small boundary property}

\author{David Kerr}
\address{David Kerr,
Mathematisches Institut,
Universit\"at M\"unster, Einsteinstr.\ 62, 48149 M\"unster, Germany}
\email{kerrd@uni-muenster.de}

\author{Grigoris Kopsacheilis}
\address{Grigoris Kopsacheilis,
Mathematisches Institut, 
Universit\"at M\"unster, Einsteinstr.\ 62, 48149 M\"unster, Germany}
\email{gkopsach@uni-muenster.de}

\author{Spyridon Petrakos}
\address{Spyridon Petrakos,
Mathematisches Institut, 
Universit\"at M\"unster, Einsteinstr.\ 62, 48149 M\"unster, Germany}
\email{spetrako@uni-muenster.de}

\date{August 30, 2024}

\begin{abstract}
We show that, for every minimal action of a countably infinite discrete group 
on a compact metrizable space, if
the extreme boundary of the simplex of invariant Borel probability measures 
is closed and has finite covering dimension then the action has the small boundary property.
\end{abstract}

\maketitle

\section{Introduction}

A basic principle in dynamics  
is that infinite groups have a complexity-lowering effect when
they act by measure-preserving or continuous transformations.
The paradigm of this regularization phenomenon is entropy,
which expresses, logarithmically, the average ``cardinality'' of a space 
under the dynamics when viewed at fixed scales (this requires an averaging mechanism
built into the group itself, in the form of amenability or soficity).
For free actions on spaces that are diffuse or have large covering dimension 
the entropy can easily be finite, and even zero.
In topological dynamics this kind of regularization also occurs at the dimensional level
and is given numerical expression through the mean dimension of the action.
Just as entropy obstructs embeddability into shifts over finite alphabets whose
cardinality is of lower logarithmic value, mean dimension obstructs embeddability
into shifts over cubes whose dimension is of lower value.

A condition closely related to but formally quite different from having mean dimension zero
is the small boundary property (SBP), which asks for a basis of opens sets whose
diameters are null for every invariant Borel probability measure (or equivalently
for every ergodic invariant Borel probability measure). 
Unlike entropy or mean dimension
this definition does not require any averaging process and makes sense
for all acting groups, but it implies mean dimension zero when the group is amenable
(Theorem~5.4 of \cite{LinWei00}) and mean dimension zero or $-\infty$ 
for every sofic approximation sequence
when the group is sofic (Theorem~8.2 of \cite{Li13}) and is known to be equivalent to 
mean dimension zero for $\Zb^d$-actions
with the marker property, which includes minimal $\Zb^d$-actions (Corollary~5.4 of \cite{GutLinTsu16}).
Moreover, when the group is amenable and the action is free
the SBP is equivalent, by an Ornstein--Weiss-type tiling argument, to 
almost finiteness in measure, which means that, up to a remainder that is
uniformly small on all invariant Borel probability measures, the space can be almost covered
by finitely many disjoint open towers with F{\o}lner shapes and levels of small diameter \cite{KerSza20} 
(see Section~\ref{S-preliminaries} for more details).

By their very definition in terms of averages over F{\o}lner sets or
sofic approximations, topological entropy and mean dimension have a measure-theoretic aspect
and are connected to the existence and distribution of probability measures
that are invariant under the action.
Given that topological entropy captures the exponential growth of the number
of partial orbits over F{\o}lner sets (or the number of dynamical models for sofic approximations)
that are distinguishable at fixed scales, and that such F{\o}lner partial orbits (or sofic models)
support uniform probability measures that are approximately invariant and 
hence weak$^*$ close to Borel probability measures that are actually invariant,
it comes as a surprise that, as the Jewett--Krieger theorem demonstrates,
nonzero and even infinite entropy is compatible with the action having only one 
invariant Borel probability measure. 
On the other hand, having nonzero mean dimension
or failing to have the SBP does have consequences for the size and geometry
of the space of invariant Borel probability measures, which brings us to the subject of the present paper.

A simple cardinality argument using balls shows that if $G\curvearrowright X$ is a continuous action 
on a compact metrizable space such that the simplex $M_G (X)$ of invariant Borel probability measures
has countable extreme boundary then the action must have the SBP.
In \cite{Ma19} Ma showed that the SBP also holds whenever $G$ is countable and amenable,
$G\curvearrowright X$ is free and minimal,
and the extreme boundary of the simplex $M_G (X)$ is compact (i.e., $M_G (X)$ is a Bauer simplex)
and zero-dimensional. 
Ma's approach makes use of noncommutative C$^*$-algebras
and was inspired by Toms, White, and Winter's proof in \cite{TomWhiWin15} of the
Toms--Winter conjecture (more precisely the implication from
strict comparison to $\sZ$-stability) in the case that the tracial state space is a Bauer simplex
whose extreme boundary has finite covering dimension, a result that was also
independently proved using similar ideas by Kirchberg and R{\o}rdam \cite{KirRor14}
and by Sato \cite{Sat12} and builds
on Matui and Sato's breakthrough in the case of finitely many extremal tracial states \cite{MatSat12}.
In this paper we take further inspiration from \cite{TomWhiWin15} as a model for 
how to proceed in the dynamical setting, in part by recasting
some of its arguments in a more purely dynamical form than what is done in \cite{Ma19},
so as to establish the following generalization of Ma's theorem.

\begin{theoremi}\label{T-main}
Let $G\curvearrowright X$ be a minimal action of a countably infinite discrete group
on a compact metrizable space. 
Suppose that $M_G (X)$ is a Bauer simplex whose extreme boundary
has finite covering dimension. Then the action has the SBP.
\end{theoremi}

One significant difference from the results of Toms--White--Winter and Ma 
is that we have been able to dispense with the hypotheses 
of amenability (or nuclearity, in C$^*$-algebraic terms) and freeness, 
and so the generalization of \cite{Ma19} actually goes in a couple of different directions.
Indeed we do not need to subject the group
to the type of approximate centrality condition that is crucial in \cite{TomWhiWin15} and that
would seem to call for F{\o}lnerness of the tower shapes (compare for example Section~9 of \cite{KerSza20}). 
It might be noted, however, that nuclearity is still
implicitly playing a role on the side of the space, via our bump function constructions
in the proof of Lemma~\ref{L-function}.

What Toms, White, and Winter actually establish is a uniform McDuff property 
(Theorem~4.6 of \cite{TomWhiWin15}; see also Sections~4 and 5 of \cite{CasEviTikWhi22}) that, 
by the work of Matui and Sato in \cite{MatSat12},
implies the passage from strict comparison to $\sZ$-stability in the Toms--Winter conjecture.
This property is analogous to almost finiteness in measure and hence, in the case of amenable groups,
to the SBP. Just as the uniform McDuff property (and even the weaker uniform property $\Gamma$ \cite{CasEviTikWhiWin21}) 
implies the equivalence of strict comparison and $\sZ$-stability in the 
Toms--Winter conjecture \cite{CasEviTikWhi22} (the backward implication being a general fact \cite{Ror04}), the SBP implies,
in the amenable case, the equivalence of comparison 
and almost finiteness (Theorem~6.1 in \cite{KerSza20}), and so we can derive the following corollary
(cf.\ Theorem~3.5 in \cite{Ma19}).
For details on comparison and almost finiteness see Section~\ref{S-preliminaries}.
Almost finiteness differs from its ``in measure'' version by requiring the remainder
to be small in a topological sense. 
For free minimal actions of countably infinite amenable groups on compact metrizable spaces,
it implies that the associated crossed product C$^*$-algebra is $\sZ$-stable \cite{Ker20} and
hence falls under the scope of the classification theorem 
for simple separable $\sZ$-stable nuclear C$^*$-algebras satisfying the UCT 
(such C$^*$-algebras we call ``classifiable'' for short) \cite{GonLinNiu20,EllGonLinNiu15,TikWhiWin17}.

\begin{corollaryi}\label{C-equivalence}
Let $G\curvearrowright X$ be a free minimal action of a countably infinite amenable discrete group
on a compact metrizable space. Suppose that $M_G (X)$ is a Bauer simplex 
whose extreme boundary has finite covering dimension. 
Then the action is almost finite if and only if it has comparison.
\end{corollaryi}

In \cite{Nar22} Naryshkin showed that, when $G$ is a finitely generated group of polynomial growth,
every minimal action $G\curvearrowright X$ on a compact metrizable space has comparison.
Together with Corollary~\ref{C-equivalence} this yields the following.

\begin{corollaryi}
Let $G\curvearrowright X$ be a free minimal action of an infinite finitely generated group of polynomial growth
on a compact metrizable space. Suppose that $M_G (X)$ is a Bauer simplex 
whose extreme boundary has finite covering dimension. 
Then the action is almost finite and the crossed product $C(X)\rtimes G$ is $\sZ$-stable and classifiable.
\end{corollaryi}

Elliott and Niu have also developed an approach to Theorem~\ref{T-main} in the case of free actions
and amenable acting groups by passing through a version of property Gamma for inclusions at the level of 
the crossed product \cite{EllNiu24}. This has enabled them to show that if a free minimal action of
an amenable group on a compact metrizable space has the uniform Rokhlin property
\cite{Niu22} then the hypothesis that $M_G (X)$ is a Bauer simplex is already in itself sufficient to imply the SBP.
The uniform Rokhlin property is satisfied by free actions of $\Zb^d$ (as follows from \cite{Gut11},
which establishes the stronger topological Rokhlin property)
but it remains unclear in what generality it holds for other acting groups.

Given a simple unital stably finite separable C$^*$-algebra $A$,
a C$^*$-diagonal $B\subseteq A$ \cite{Kum86}, and a countable group $G$ of unitary normalizers
of $B$ such that $G$ and $B$ together generate $A$, one can consider the induced minimal 
action of $G\curvearrowright X$ on the spectrum of $B$ (for minimality see Section~10.3 of \cite{SimSzaWil20}), 
in which case $M_G (X)$ identifies with the simplex of tracial states on $A$ via the conditional expectation onto $B$ 
(see Section~3 of \cite{CryNag16}). 
Theorem~\ref{T-main} shows that if the trace simplex of $A$ is Bauer with
finite-dimensional extreme boundary then the action $G\curvearrowright X$ has the SBP.
In general $G\curvearrowright X$ need not have the SBP, as the crossed products of the
minimal $\Zb$-actions without the SBP in \cite{LinWei00} demonstrate. But these crossed products are 
not $\sZ$-stable and hence not classifiable \cite{GioKer10}. It is an interesting question
whether $G\curvearrowright X$ will always have the SBP when $A$ is classifiable (see for example \cite{LiaTik22}).

We begin in Section~\ref{S-preliminaries} with some general definitions and notation. 
Section~\ref{S-affine} is devoted to a dynamical analogue of a C$^*$-algebraic affine function realization result
that appears as Theorem~9.3 in the paper \cite{Lin07} by Lin and was employed in both \cite{TomWhiWin15} and \cite{Ma19}.
Like Theorem~9.3 in \cite{Lin07}, this is based on work of Cuntz and Pedersen on equivalence and traces
in \cite{CunPed79}. It is pointed out by Cuntz and Pedersen in Section~8 of \cite{CunPed79} that similar
arguments also work in an equivariant setting such as ours, and so we are merely making explicit in the commutative case
some of the technical details on dynamical equivalence that were left to the reader in \cite{CunPed79} 
and then putting them together as in \cite{Lin07} to obtain the realization result, which we record as Theorem~\ref{T-realization}. 
With this at hand we then proceed to the proof of Theorem~\ref{T-main}, via a series of lemmas, in Section~\ref{S-proof}.
At the end of Section~\ref{S-proof} we also give some examples illustrating the application of Theorem~\ref{T-main}.
\medskip

\noindent{\it Acknowledgements.}
The authors were supported by the Deutsche Forschungsgemeinschaft
(DFG, German Research Foundation) under Germany's Excellence Strategy EXC 2044-390685587,
Mathematics M{\"u}nster: Dynamics--Geometry--Structure, and by the SFB 1442 of the DFG. The second named author was also supported by ERC Advanced Grant 834267 -- AMAREC.

\section{Preliminaries}\label{S-preliminaries}

For a bounded Borel function $f$ on a compact metrizable space $X$ and a Borel
probability measure $\mu$ on $X$ we usually write $\mu (f)$ for the integral $\int_X f\, d\mu$.

Let $G$ be a countable discrete group and $G\curvearrowright X$ a continuous action
on a compact metrizable space. We write $M_G (X)$ for the convex set of all $G$-invariant
Borel probability measures on $X$ equipped with the weak$^*$ topology, under which
it is a Choquet simplex.
The extreme boundary of $M_G (X)$ is the set of ergodic measures and is written $M_G^\erg (X)$.

A Choquet simplex is said to be a {\it Bauer simplex} if the set of its extreme points is closed.
This property is characterized by the fact that every continuous function on the 
extreme boundary uniquely extends to a continuous affine function on the whole simplex.
It obviously holds when there are only finitely many extreme points,
but can already fail when the set of extreme points is countably infinite.
At the opposite end from Bauer simplices is the Poulsen simplex, 
characterized by the fact that the extreme points are dense. 

For many amenable groups $G$, including $\Zb$,
every Choquet simplex can be realized as $M_G (X)$
for some free minimal subshift action $G\curvearrowright X$ \cite{Dow91,CecCor19}.
Glasner and Weiss showed that if $G$ does not have property (T) (which is the case for example
if $G$ is infinite and amenable) and $G\curvearrowright Y^G$ is a nontrivial shift action
then $M_G (X)$ is the Poulsen simplex, while if $G$ has property (T) then for
every action $G\curvearrowright X$ the simplex $M_G (X)$, when nonempty, is Bauer \cite{GlaWei97}.

The action $G\curvearrowright X$ has the {\it small boundary property (SBP)}
if $X$ has a basis of open sets whose boundaries are null for $G$-invariant 
Borel probability measures. The SBP holds if and only if
for all disjoint closed sets $C,D\subseteq X$ and $\delta > 0$ there exists an 
open set $U\subseteq X$ such that $C\subseteq U \subseteq X\setminus D$ 
and $\mu (\partial U) \leq \delta$ for every $\mu\in M_G (X)$.
This is the equivalence (i)$\Longleftrightarrow$(v) in Theorem~5.5 of \cite{KerSza20},
and although the group $G$ is assumed to be amenable there the argument
can be rewritten so as to work for general $G$ by replacing the use
of Proposition~3.4 in \cite{KerSza20} with the fact that for every $\eps > 0$ and
$K\subseteq X$ satisfying $\mu(K)<\eps$ for all $\mu\in M_G(X)$
then there is an open set $W\supseteq K$ such that $\mu(W)<\eps$ for all $\mu\in M_G(X)$
(one can apply the portmanteau theorem to see this).
The SBP is automatic when $G$ is countably infinite, $X$ has finite covering dimension and the action is free \cite{Lin95,Sza15}.
Examples of free minimal $\Zb$-actions without the SBP can be found in \cite{LinWei00}.

For sets $U,V\subseteq X$ we say that $U$ is {\it subequivalent} to $V$ and
write $U\prec V$ to mean that for every compact set $A\subseteq U$
there are finitely many open subsets $U_1 , \dots , U_n$ of $X$ covering $A$
and $s_1 , \dots , s_n \in G$ such that the set $s_i A_i$ for $i=1,\dots ,n$ are
pairwise disjoint subsets of $V$.
The action has {\it comparison} if $U\prec V$ for all nonempty open sets $U,V\subseteq X$ satisfying
$\mu (U) < \mu (V)$ for every $\mu\in M_G (X)$.

A {\it tower} is a pair $(S,B)$ where $S$ is a finite subset of $G$ (the {\it shape} of the tower)
and $B$ is a subset of $X$ (the {\it base} of the tower) such that the sets $sB$ for $s\in S$ 
(the {\it levels} of the tower) are pairwise disjoint. When convenient we can reparametrize the tower $(S,B)$
so that $e\in S$ (assuming $S$ is nonempty) by choosing a $t\in S$ and considering instead $(St^{-1} , tB)$.
The tower is {\it open}, {\it Borel}, etc.\
if the levels are open, Borel, etc. We may also use the word tower to refer to the set $SB$ when convenient.
A {\it castle} is a finite collection $\{ (S_i,B_i) \}_{i\in I}$
of towers such that the sets $S_i B_i$ for $i\in I$ are pairwise disjoint.
The castle is {\it open}, {\it Borel}, etc.\ if all of the towers are open, Borel, etc.

The action $G\curvearrowright X$ is {\it almost finite in measure} if for every finite set $K\subseteq G$
and $\delta > 0$ there is an open castle $\{ (S_i , V_i ) \}_{i\in I}$ such that 
\begin{enumerate}
\item $|tS_i \triangle S_i |/|S_i| < \delta$ for every $t\in K$ and $i\in I$, 

\item the levels of the castle have diameter less than $\delta$,

\item $\mu (X\setminus\bigsqcup_{i\in I} S_i V_i ) < \delta$ for every $\mu\in M_G (X)$.
\end{enumerate}
If instead of (iii) we require the stronger condition that
there exist $S_i' \subseteq S_i$ with $|S_i' | < \delta |S_i |$ such that
$X\setminus\bigsqcup_{i\in I} S_i V_i \prec \bigsqcup_{i\in I} S_i' V_i$,
then we say that the action is {\it almost finite}. Note that the F{\o}lner condition in (i) 
forces $G$ to be amenable.
If the action $G\curvearrowright X$ is free then it is 
almost finite if and only if it is almost finite in measure and has comparison
(Theorem~6.1 of \cite{KerSza20}),
and if moreover $G$ is amenable then almost finiteness in measure is equivalent
to the SBP (Theorem~5.6 of \cite{KerSza20}).

\section{Affine function realization}\label{S-affine}

We record in Theorem~\ref{T-realization} the dynamical analogue of the C$^*$-algebraic affine function realization result
that appears as Theorem~9.3 in \cite{Lin07}. Like the latter, Theorem~\ref{T-realization} follows 
from the theory of equivalence developed by 
Cuntz and Pedersen in \cite{CunPed79}, only this time in the equivariant setting, which was discussed in Section~8 of \cite{CunPed79}. 
The technicalities in the equivariant case are analogous to those in the non-equivariant case and were omitted in \cite{CunPed79}. 
For the convenience of the reader we will supply the details, following \cite{CunPed79}, 
for what we need to reach Theorem~\ref{T-realization}.

Throughout this section $G\curvearrowright X$ is a minimal action of a countable discrete group on a compact metrizable space.
From the action $G\curvearrowright X$ we obtain an action of $G$ on $C(X)$ 
by the formula $(sf)(x) = f(s^{-1} x)$ for all $f\in C(X)$, $s\in G$, and $x\in X$.
On the set $C(X,\mathbb{R})_+$ of nonnegative functions in $C(X,\mathbb{R})$
we define a relation $\sim$ as follows: we say that $f\sim g$ whenever there exist sequences 
$(h_n)$ in $C(X,\mathbb{R})_+$ and $(s_n)$ in $G$ such that $f=\sum_nh_n$ and $g=\sum_ns_n h_n$,
where the sums are uniformly convergent.
We also define a relation $\prec$ on $C(X,\mathbb{R})_+$ by declaring that $f\prec g$ whenever there exists some 
$f'\in C(X,\mathbb{R})_+$ such that $f\sim f'\le g$.  

\begin{proposition}
The relations $\prec$ and $\sim$ are transitive.
\end{proposition}

\begin{proof}
Let us show that $\prec$ is transitive; the transitivity of $\sim$ is proved using the same argument. Let $f,g,h\in C(X,\mathbb{R})_+$ be such that $f\prec g$ and $g\prec h$. Then we can write
$f=\sum_ia_i$ and $\sum_is_ia_i \le g=\sum_jb_j$ and $\sum_jt_jb_j\le h$ 
where $(a_i)_i,(b_j)_j$ are sequences in $C(X,\mathbb{R})_+$ and $(s_i)_i,(t_j)_j$ are sequences in $G$. 
Writing $U_i$ for the open support of $a_i$, for $x\in X$ we set
\begin{gather*}
c_{i,j}(x)=\begin{cases}\frac{a_i(x)b_j(s_ix)}{g(s_ix)}, & x \in U_i \\ 0, &\text{otherwise.}\end{cases}
\end{gather*}
This is well defined, for if $a_i(x)\ne0$ then $g(s_ix)\ge\sum_ka_k(s_k^{-1}s_ix)\ge a_i(x)>0$. 
To see that $c_{i,j}$ is continuous, it is enough to show, given an $x\in\partial U_{i}$ and a sequence $(x_n)$
in $U_i$ with $x_n\to x$, that $c_{i,j}(x_n)\to 0$, and this follows from the fact that
$a_i(x_n)\to 0$ and $b_j\le g$.

Next we remark that for $x\in U_i$ we have $\sum_jc_{i,j}(x)=a_i(x)$, while for $x\not\in U_i$ we have
$0=\sum_jc_{i,j}(x)=a_i(x)$. So in general $\sum_jc_{i,j}(x)=a_i(x)$. The convergence is uniform by Dini's theorem, 
and thus $\sum_i\sum_jc_{i,j}=\sum_ia_i=f$.

Now let $j\in\Nb$ and $x\in X$. If $g(t_j^{-1}x)=0$, then $c_{i,j}(s_i^{-1}t_j^{-1}x)=0$ for all $i\in\Nb$, so that 
$\sum_ic_{i,j}(s_i^{-1}t_j^{-1}x)=0$. If $g(t_j^{-1}x)\ne0$, then $c_{i,j}(s_i^{-1}t_j^{-1}x) = \frac{a_i(s_i^{-1}t_j^{-1}x)b_j(t_j^{-1}x)}{g(t_j^{-1}x)}$, so that $\sum_ic_{i,j}(s_i^{-1}t_j^{-1}x)\le b_j(t_j^{-1}x)$. In either case,
\begin{gather*}
\sum_j\sum_it_js_ic_{i,j}(x)\le\sum_jt_jb_j(x)\le h(x),
\end{gather*}
which shows that the series $\sum_j\sum_it_js_ic_{i,j}$ converges pointwise, 
and hence also uniformly by Dini's theorem, to a continuous function dominated by $h$. 
\end{proof}

The following properties are readily checked and will often be used without comment.
For (v) one verifies by induction that
$na+c\sim nb+d$, and then applies (iv).
\begin{enumerate}
\item If $(f_n)$ and $(g_n)$ are sequences such that $f_n \sim g_n$ for every $n$ and 
the series $\sum_nf_n=f$ and $\sum_ng_n=g$ converge uniformly, then $f\sim g$.

\item If $f_1 \prec g_1$ and $f_2 \prec g_2$ then $f_1 + f_2 \prec g_1 + g_2$.

\item If $f\leq g$ or $f\sim g$ then $f\prec g$.

\item If $f\sim g$ (resp.\ $f\prec g$) then $\lambda f \sim \lambda g$ (resp.\ $\lambda f \prec \lambda g$) for all $\lambda\geq 0$.

\item If $a,b,c,d\in C(X,\mathbb{R})_+$ are such that $c\sim d$ and $a+c\sim b+d$, then for every $n\in\Nb$ 
one has $a+\frac{1}{n}c\sim b+\frac{1}{n}d$.
\end{enumerate}

Next we define
\[
I=\{f-g: f,g\in C(X,\mathbb{R})_+,\, f\sim g\} .
\] 

\begin{lemma}\label{L-closed subspace} 
$I$ is a closed linear subspace of $C(X,\Rb )$.
\end{lemma}

\begin{proof}
That $I$ is a linear subspace follows from the basic properties of $\sim$. 
To show that $I$ is closed, we first verify, for a given $f\in I$ and $\varepsilon>0$,
that there are $g,h\in C(X,\mathbb{R})_+$ such that $g\sim h$ and $f=g-h$ and $\|g+h\|\le \|f\|+\varepsilon$.
By the definition of $I$ we can write $f=a-b$ where $a,b\in C(X,\mathbb{R})_+$ are such that $a\sim b$. 
Writing $f_+$ (resp.\ $f_-$) for the positive (resp.\ negative) part of $f$, we have $f_+ + b = f_- + a$, as $f=f_+ - f_-$. By property (v) we have $f_+ + \frac{1}{n}b\sim f_-+\frac{1}{n}a$ for every $n$.
Set $g_n=f_+ + \frac{1}{n}(a+b)$ and $h_n=f_- + \frac{1}{n}(a+b)$. Then $g_n-h_n=f$ and
\begin{gather*}
g_n = \Big(f_+ + \frac{1}{n} b\Big) + \frac{1}{n} a \sim f_- + \frac{1}{n} a + \frac{1}{n} a \sim f_- +\frac{1}{n} a + \frac{1}{n}b = h_n,
\end{gather*}
while for large enough $n$ we have $\|g_n+h_n\|=\| |f| + \frac{2}{n}(a+b)\|<\|f\|+\varepsilon$, in which case
$g_n$ and $h_n$ fulfill the desired conditions.

Now let $(f_n)$ be a sequence in $I$ with $f_n\to f\in C(X,\mathbb{R})$ and let us show that $f\in I$.
By passing to a subsequence if necessary, we may assume without loss of generality that $\|f_{n+1}-f_n\|<2^{-n}$ for each $n$. Since $I$ is a linear subspace and hence $f_{n+1}-f_n\in I$, by the first paragraph
we can find for each $n$ elements $g_n,h_n\in C(X,\mathbb{R})_+$ such that $g_n\sim h_n$ and
$f_{n+1}-f_n=g_n-h_n$ and $\|g_n+h_n\|<2^{-n}$. Since $\max\{\|g_n\|,\|h_n\|\}<2^{-n}$, the series $g:=\sum_ng_n$ and $h:=\sum_nh_n$ converge uniformly. Also, $g\sim h$ by property (i).
Finally, we observe that
\begin{gather*}
f-f_1=\sum_n(g_n-h_n)=g-h
\end{gather*}
and therefore $f=(g-h)+f_1\in I$.
\end{proof}

Having shown $I$ to be a closed linear subspace, we define $V$ to be the quotient space $C(X,\mathbb{R})/I$
with the quotient norm $\|\! \cdot\! \|_V$.
The quotient map $C(X,\mathbb{R})\to V$ will be written $g\mapsto\overline{g}$.
The dual $V^*$ of $V$ identifies in the obvious way with the annihilator of $I$, i.e., with the space of (signed) Borel measures 
that satisfy $\mu(f)=\mu(g)$ whenever $f\sim g$.

\begin{lemma}\label{L-constant}
Let $g\in C(X,\mathbb{R})_+ \setminus \{ 0 \}$. Then there is a constant $L(g)>0$ such that $h\prec \|h\| L(g) g$ for all $h\in C(X,\mathbb{R})_+$.
\end{lemma}

\begin{proof}
Let $U$ be the open support of $g$, which is nonempty. By minimality and compactness we can find $s_1,\dots,s_n\in G$ 
such that $\bigcup_{i=1}^ns_iU=X$, so that there exists a $\theta>0$ for which $\sum_{i=1}^ns_ig\ge\theta$. For $h\in C(X,\mathbb{R})_+$ we have
\[
h\le \|h\|1\le \frac{\|h\|}{\theta}\sum_{i=1}^ns_ig\sim \frac{n}{\theta}\|h\|g,
\]
and so we may take $L(g)=n/\theta$ by the transitivity of $\prec$.
\end{proof}

\begin{lemma}\label{L-open}
The image of $C(X,\mathbb{R})_+\setminus\{0\}$ under the quotient map $C(X,\mathbb{R})\to V$ is an open set. 
\end{lemma}

\begin{proof}
Let $z$ be an element in the image of $C(X,\mathbb{R})_+\setminus\{0\}$, 
say $z=\overline{g}$ where $0\ne g\in C(X,\mathbb{R})_+$ 
is a representative. With $L(g)$ as in Lemma~\ref{L-constant},
we will show that the open ball centred at $z$ of radius $\frac{1}{2}L(g)^{-1}$ is contained in the image of $C(X,\mathbb{R})_+\setminus\{0\}$. Let $\overline{h}$ be an element of $V$ with $\|\overline{h}-\overline{g}\|<\frac{1}{2}L(g)^{-1}$. 
Take $k\in (g-h)+I$ with $\|k\|<\frac{1}{2}L(g)^{-1}$. By the definition of $L(g)$ we have $|k|\prec \|k\| L(g) g\le \frac{1}{2}g$,
so that there exists a $w\in C(X,\mathbb{R})_+$ such that $|k|\sim w\le \frac{1}{2}g$. Set $k' = k + w - |k|$. We have $k'-k\in I$, and therefore the class of $h$ in $V$ is the same as the class of $g-k'$. Since $g-k'\ge g-w\ge\frac{1}{2}g\ne0$, this shows that $\overline{h}$ belongs to the image of $C(X,\mathbb{R})_+\setminus\{0\}$.
\end{proof}

\begin{lemma}\label{L-equality}
$\{ g\in C(X,\mathbb{R}): \mu(g)>0 \text{ for all }\mu\in M_G(X)\}=(C(X,\mathbb{R})_+\setminus\{0\})+I$.
\end{lemma}

\begin{proof}
Write $A$ for the set on the left-hand side.
The inclusion $\supseteq$ is a consequence of minimality. 
For the reverse inclusion, since $A=A+I$ it suffices to show that the images of $A$ and
$C(X,\mathbb{R})_+\setminus\{0\}$ under the quotient map $C(X,\mathbb{R})\to V$ are equal.
For brevity write $W$ for the second of these images, which is open by Lemma~\ref{L-open}.
Suppose that the two images are not equal. Then there is a $g\in C(X,\mathbb{R})$ such that $\mu(g)>0$ for all $\mu\in M_G(X)$ 
but $\overline{g}\not\in W$. Since $W$ and $\{\overline{g}\}$ are both convex and $W$ is open, by the Hahn--Banach 
saparation theorem there is a $\varphi\in V^*$ such that $\varphi(\overline{g})\le t<\varphi(\overline{h})$ for all $h\in C(X,\mathbb{R})_+\setminus\{0\}$, which means that there is a 
signed invariant Borel measure $\mu$ such that $\mu(g)\le t<\mu(h)$ for all $h\in C(X,\mathbb{R})_+\setminus\{0\}$. 
Since $t<\mu(\varepsilon1)$ for all $\varepsilon>0$ we must have $t\le0$. At the same time, if there were 
an $h\in C(X,\mathbb{R})_+\setminus\{0\}$ such that $\mu(h)<0$ then we would have that $t<n\mu(h)$ for every $n\in\Nb$, 
a contradiction. It follows that $\mu(g) \le 0 \le \mu(h)$ for all $h\in C(X,\mathbb{R})_+ \setminus \{0\}$. 
The second of these inequalities shows that $\mu$ is a positive invariant Borel measure, while the first 
gives a contradiction to the fact that $\nu(g)>0$ for all $\nu\in M_G(X)$.
\end{proof}

\begin{lemma}\label{L-equality norm}
Let $g\in C(X,\mathbb{R})_+$. Then $\|\overline{g}\|_V=\inf\{\|g-k\|: k\in I, \, k\le g\}$. 
\end{lemma}

\begin{proof}
Put $\alpha =\inf\{\|g-k\|: k\in I, \, k\le g\}$. By definition $\|\overline{g}\|_V \le\alpha$, 
and so let us show the reverse inequality. 

First we set $\beta = \inf\{\|h\|: h\in C(X,\Rb)_+, \, g\prec h\}$ and argue that $\alpha\leq\beta$.
Given an $\varepsilon>0$, take an $h\in C(X,\Rb)_+$ with $g\prec h$ such that $\|h\|<\beta+\varepsilon$. 
Then we can find a $g'\in C(X,\mathbb{R})_+$ such that $g \sim g' \le h$. Put $k=g-g'$. Since $g\sim g'$, 
we have $k \in I$, and since $g'\ge0$ we have $k\le g$. 
Moreover $\alpha\le\|g-k\|=\|g'\|\le \|h\|<\beta+\varepsilon$, so letting $\varepsilon\to0$ yields $\alpha\le\beta$.

To conclude, let us show that $\beta\le\|\overline{g}\|_V$. Given an $\varepsilon>0$ we find a $k\in I$ with $\|g-k\|<\|\overline{g}\|_V+\varepsilon$. 
Write $g-k = (g-k)_+ - (g-k)_-$. Since $k\in I$ we can write $k=k_1-k_2$ with $k_1,k_2\in C(X,\mathbb{R})_+$ and $k_1\sim k_2$.
Then $g + (g-k)_- + k_2 = (g-k)_+ + k_1$. It follows by property (v) that for every $n\in\Nb$ we have
\[
g \le g + (g-k)_- +\frac{1}{n}k_2 \sim (g-k)_+ + \frac{1}{n}k_1 
\]
and therefore $g \prec (g-k)_++\frac{1}{n}k_1$ by the transitivity of $\prec$, so that for large enough $n$ we obtain
\begin{gather*}
\beta \le \Big\|(g-k)_++\frac{1}{n}k_1 \Big\|\le\|g-k\|+\frac{1}{n}\|k_1\|<\|\overline{g}\|_V+2\varepsilon . \qedhere
\end{gather*}
\end{proof}

\begin{theorem}\label{T-realization}
Let $G\curvearrowright X$ be a minimal action of a countably infinite discrete group
on a compact metrizable space. Suppose that $M_G (X)$ is a Bauer simplex.
Let $f$ be a strictly positive function in $C(M_G^\erg (X),\Rb)$ and let $\delta > 0$. 
Then there is a nonnegative function $h\in C(X,\Rb)$ with $\| h \| \leq \| f \| + \delta$ 
such that $\mu (h) = f(\mu )$ for all $\mu\in M_G^\erg (X)$.
\end{theorem}

\begin{proof}
Since every continuous function on the extreme boundary of a Bauer simplex extends (uniquely)
to a continuous affine function on the whole simplex,
we may assume that $f$ is a strictly positive continuous affine function on $M_G (X)$.
Since $f$ is continuous and affine it extends to an element $\tilde{f}\in V^{**}$ which is weak$^*$ continuous 
and hence belongs to $V$. It follows by Lemma~\ref{L-equality} that there exists a
$g\in C(X,\mathbb{R})_+\setminus\{0\}$ such that $\overline{g}=\tilde{f}$.  
By Lemma~\ref{L-equality norm} we can find a $k\in I$ with $k\le g$ such that $\|g-k\|<\|f\|+\delta$. 
Since $g-k\ge0$ and $g-k$ has the same class as $g$ in $V$ (which is $\tilde{f}$), we are done. 
\end{proof}

\section{Proof of Theorem~\ref{T-main}}\label{S-proof}

We begin with a series of technical lemmas required for the perturbation-and-patching argument 
in the beginning of the proof of Lemma~\ref{L-scale}.
The first one is a special case of Lemma~3.4 in \cite{TomWhiWin15}. 
For completeness we include a proof.

\begin{lemma}\label{L-sequence}
Let $X$ be a compact metrizable space and let $W$ be a set of Borel probability measures on $X$. 
Suppose that $(f_{1,n})_{n=1}^\infty,\dots,(f_{L,n})_{n=1}^\infty$ are sequences in $C(X,[0,1])$ such that 
\[
\lim_{n\to\infty}\sup_{\mu\in W}\mu(f_{l,n}f_{l',n})=0
\]
for all $l\ne l'$. Then there exist $0\le \tilde{f}_{l,n}\le f_{l,n}$ such that
\begin{enumerate}
\item $\lim_{n\to\infty}\sup_{\mu\in W}\big|\mu(f_{l,n})-\mu(\tilde{f}_{l,n})|=0$ for all $l=1,\dots,L$, and

\item $\lim_{n\to\infty}\|\tilde{f}_{l,n}\tilde{f}_{l',n}\|=0$ for all $l\ne l'$.
\end{enumerate}
\end{lemma}

\begin{proof}
For $r\in\mathbb{N}$ set $\varphi_r(t)=\min\{1,rt\}$. Then $\sup_{t\ge0}|t(1-\varphi_r(t))|=(4r)^{-1}$. 
For each $l=1,\dots,L$ set $g_{l,n}=f_{l,n}\cdot\sum_{l'\ne l}f_{l',n}$,
in which case $\lim_{n\to\infty}\sup_{\mu\in W}\mu(g_{l,n})=0$. Put 
\[
h_{l,n,r}=f_{l,n} \cdot (1-\varphi_r \circ g_{l,n})
\]
and note that
\begin{gather*}
\sup_{\mu\in W}|\mu(f_{l,n})-\mu(h_{l,n,r})|
=\sup_{\mu\in W}\mu(f_{l,n}(\varphi_r \circ g_{l,n}))
\le r\cdot\sup_{\mu\in W}\mu(g_{l,n})\xrightarrow{n\to\infty}0 
\end{gather*}
and, for $l \ne l'$,
\begin{align}\label{eq:pworthogonal}
\|h_{l,n,r} h_{l',n,r}\|
&=\|f_{l,n}f_{l',n}(1-\varphi_r \circ g_{l,n})(1-\varphi_r \circ g_{l',n})\| \\
&\le\|g_{l,n}(1-\varphi_r \circ g_{l,n}))\| 
\le \frac{1}{4r}. \notag
\end{align}
The proof now concludes with a saturation argument: for each $r\in\mathbb{N}$ 
pick an $N_r\in\mathbb{N}$ so that $\sup_{\mu\in W}|\mu(f_{l,n})-\mu(h_{l,n,r})|\le1/r$ 
for all $n\ge N_r$. We may assume that $N_1 = 1<N_2<N_3 < \dots$. 
For $n\in\mathbb{N}$, set
$r_n:=\max\{r\in\mathbb{N}: N_r\le n\}$. Clearly $N_{r_n}\le n<N_{r_n+1}$ for all $n\in\mathbb{N}$. Setting $\tilde{f}_{l,n}=h_{l,n,r_n}$ we have
\[
\sup_{\mu\in W}|\mu(f_{l,n})-\mu(\tilde{f}_{l,n})|\le \frac{1}{r_n}
\]
for all $n\in\mathbb{N}$. This and \eqref{eq:pworthogonal}, along with the fact that $r_n\to\infty$ as $n\to\infty$, conclude the proof.
\end{proof}

We now apply Lemma~\ref{L-sequence} and Theorem~\ref{T-realization} to establish the following (cf.\ Lemma~3.1 in \cite{Ma19}).

\begin{lemma}\label{L-mutual singularity}
Let $G\curvearrowright X$ be a minimal action of a countably infinite group on a compact metrizable space. 
Let $W_1,\dots, W_L$ be pairwise disjoint closed subsets of $M_G^\mathrm{erg}(X)$
and let $\eps > 0$.
Then there exist pairwise disjoint closed sets $K_1,\dots,K_L\subseteq X$ such that
\[
\mu(K_j)>1-\varepsilon
\]
for all $\mu\in W_j$ and $j=1,\dots,L$. 
\end{lemma}

\begin{proof}
Since $W_1,\dots, W_L$ are closed and pairwise disjoint, by Urysohn's lemma we can find,
for each $n\in\mathbb{N}$, functions 
$\psi_{1,n},\dots,\psi_{L,n}\in C(M_G^\mathrm{erg}(X),[1/n,1-1/n])$ with 
$\psi_{l,n}\vert_{W_l}=1-1/n$ and $\psi_{l,n}\vert_{W_{l'}}=1/n$ for $l'\ne l$. 
Since $M_G(X)$ is a Bauer simplex, we can extend these to strictly positive affine 
continuous functions on $M_G(X)$ of norm at most $1-1/n$, for which we use the same notation. 
By Theorem~\ref{T-realization} there exist $f_{l,n}\in C(X,[0,1])$ such that 
$\mu(f_{l,n})=\psi_{l,n}(\mu)$ for all $\mu\in M_G(X)$. 
The condition in Lemma~\ref{L-sequence} is clearly satisfied for $W=\bigsqcup_{i=1}^L W_i$ there, and so by replacing 
$f_{l,n}$ with $\tilde{f}_{l,n}$ from the conclusion of that lemma
we may also assume that $\lim_{n\to\infty}\|f_{l,n}f_{l',n}\|=0$.

Choose $n\in\Nb$ large enough so that
\begin{enumerate}
\item $\mu(f_{l,n})>1-\varepsilon$ for all $\mu\in W_l$, and

\item $\|f_{l,n}f_{l'n}\|<\varepsilon^2$ for all $l\ne l'$. 
\end{enumerate}
Setting $K_l=\{x\in X: f_{l,n}(x)\ge\varepsilon\}$ for each $l$, 
we obtain closed sets which are pairwise disjoint due to (ii). Moreover, for $\mu\in W_l$
we have $\int_{X\setminus K_l}f_{l,n}\, d\mu\le \varepsilon\cdot\mu(X\setminus K_l)\le\varepsilon$ and hence, by (i), 
\begin{gather*}
\mu(K_l)>\int_{K_l}f_{l,n}\, d\mu
=\mu(f_{l,n})-\int_{X\setminus K_l}f_{l,n}\, d\mu>1-2\varepsilon. \qedhere
\end{gather*}
\end{proof}

The next lemma is the well-known comparison property for ergodic probability-measure-preserving actions. 

\begin{lemma}\label{L-comparison}
Let $G\curvearrowright (X,\mu )$ be an ergodic measure-preserving action of a countable group on a probability space.
Let $A,B\subseteq X$ be measurable sets with $\mu(A)=\mu(B)$. Then
there exist a collection $\{ A_s \}_{s\in G}$ of pairwise disjoint measurable subsets of $A$
indexed by $G$ (some of which may be empty) such that $\mu (A\setminus\bigsqcup_{s\in G} A_s ) = 0$
and the sets $sA_s$ for $s\in G$ are pairwise disjoint and contained in $B$.
\end{lemma}

\begin{proof}
Fix an enumeration $s_1 , s_2 , \dots$ of $G$ and recursively define pairwise disjoint
measurable subsets $A_1 , A_2 , \dots$ of $A$ by setting $A_1 = A\cap s_1^{-1} B$ and
$A_n = (A\setminus \bigsqcup_{i=1}^{n-1} A_i ) \cap s_n^{-1}(B\setminus\bigsqcup_{i=1}^{n-1}s_iA_i)$ for $n>1$.
Then the sets $s_n A_n$ for $n\ge1$ are pairwise disjoint and contained in $B$.
Moreover, the complement $C = A\setminus\bigsqcup_{n\ge1} A_n$ satisfies $\mu (C) = 0$,
for otherwise the measure of the set $D = B\setminus \bigsqcup_{n\ge1} s_n A_n$,
being equal to $\mu (C)$, would also be nonzero, in which case ergodicity and the $G$-invariance of 
the set $GC$ would
imply the existence of an $n_0 \in\Nb$ such that $\mu (s_{n_0} C\cap D) > 0$, contradicting 
our choice of the set $A_{n_0}$. 
\end{proof}

\begin{lemma}\label{L-castle}
Let $G\curvearrowright (X,\mu )$ be an ergodic measure-preserving action of a countable group on a probability space.
Let $q\in\mathbb{N}$ and let $A_1,\dots, A_q \subseteq X$ be pairwise disjoint measurable sets of equal measure.
Let $\eps > 0$. Then there exists a finite collection $\{B_j \}_{j\in J}$ of pairwise disjoint measurable
subsets of $A_1$ and sets $S_j \subseteq G$ of cardinality $q$ such that
$\{(S_j ,B_j )\}_{j\in J}$ is a castle and
\begin{enumerate}
\item for every $j\in J$ there is an enumeration $s_{j,1} , \dots , s_{j,q}$ of $S_j$
such that $s_{j,k}B_j\subseteq A_k$ for all $k=1,\dots,q$, and

\item $\mu(\bigsqcup_{j\in J}S_j B_j )\geq\mu(\bigsqcup_{k=1}^q A_k)-\varepsilon$.
\end{enumerate}
\end{lemma}

\begin{proof}
By Lemma~\ref{L-comparison}, 
for every $k=1,\dots ,q$ there exists a collection $\{ A_{k,s} \}_{s\in G}$ of pairwise disjoint 
measurable subsets of $A_1$ indexed by $G$ such that $\mu (A_1\setminus\bigsqcup_{s\in G} A_{k,s} ) = 0$
and the sets $sA_{k,s}$ for $s\in G$ are pairwise disjoint and contained in $A_k$.
Writing $W$ for the collection of $q$-tuples of distinct elements of $G$,
for every $w=(s_1 , \dots , s_q )$ we set 
$B_w = \bigcap_{k=1}^q A_{k,s_k}$ and $S_w = \{ s_1 , \dots , s_q \}$,
in which case $|S_w | = q$ and $s_k B_w\subseteq A_k$ for all $k=1,\dots,q$.
The sets $B_w$ partition a conull subset of $A_1$, and so by the countable additivity of $\mu$
we can find a finite subcollection $W_0 \subseteq W$ such that 
$\mu(\bigsqcup_{w\in W_0}B_w )\geq\mu(A_1)-\varepsilon /q$
and hence $\mu(\bigsqcup_{w\in W_0}S_w B_w )\geq q (\mu (A_1 ) - \eps /q)
= \mu(\bigsqcup_{k=1}^q A_k)-\varepsilon$.
Thus the sets $S_w$ and $B_w$ for $w\in W_0$ will do the job.
\end{proof}

\begin{lemma}\label{L-qshaped castles for individual measures}
Let $G\curvearrowright X$ be a minimal action of a countably infinite discrete group
on an infinite compact metrizable space. Let $\mu\in M_G^\erg (X)$. Let $q\in\Nb$. Let $\Omega$
be a finite subset of $C(X,[0,1])$. Let $\delta > 0$.
Then there is an open castle $\{ (S_j , V_j ) \}_{j\in J}$ 
such that, for each $j\in J$,
\begin{enumerate}
\item the shape $S_j$ has cardinality $q$,

\item the diameters of the sets $h(S_jV_j )\subseteq \Rb$ for $h\in\Omega$
are all at most $\delta$,

\item $\mu ( \bigsqcup_{j\in J} S_jV_j ) \geq 1-\delta$.
\end{enumerate}
\end{lemma}

\begin{proof}
By compactness and continuity we can find a finite open cover $\mathcal{U}$ of $X$ so that for every 
$U\in\mathcal{U}$ and $h\in\Omega$ the image $h(U)$ 
has diameter at most $\delta/2$ in $\Rb$. 
Applying the standard recursive disjointification process with respect to some fixed enumeration of the
elements of $\mathcal{U}$, we obtain a finite Borel partition $\{B_i\}_{i\in I}$ of $X$
such that $h(B_i)$ has diameter at most $\delta/2$ for all $h\in\Omega$ and $i\in I$.

Since $G$ and $X$ are infinite and the action is minimal, the measure $\mu$ is atomless.
Thus each $B_i$ can be written as a disjoint union of $q$ sets, say $B_i^{(l)}$ for $l=1,\dots ,q$,
which each have measure $\mu(B_i)/q$. 
For each $i\in I$ we apply Lemma~\ref{L-castle} to $\{B_i^{(l)}\}_{l=1}^q$ so as
to obtain a Borel castle $\{(S_{i,j},B_{i,j})\}_{j\in J_i}$ such that 
$\bigsqcup_{j\in J_i} S_{i,j} B_{i,j} \subseteq B_i$, 
each shape $S_{i,j}$ has cardinality $q$, and 
$\mu(\bigsqcup_{j\in J_i}S_{i,j}B_{i,j})\ge\mu(B_i)-\delta /(2|I|)$. 
Putting these $|I|$ many castles together, we obtain a Borel castle $\{(S_j , E_j )\}_{j\in J}$ 
such that each shape has cardinality $q$, the measure of $\bigsqcup_{j\in J}S_jE_j$ is at least $1-\delta /2$, 
and each tower is entirely contained in some $B_i$. 

Using the regularity of $\mu$, for each tower base $E_j$ we can choose a compact subset $K_j\subseteq E_j$ 
so that $\mu(K_j)>\mu(E_j)-\delta/(2q|J|)$. 
Since $\{(S_j ,K_j )\}_{j\in J}$ is a closed castle, there is a positive constant bounding from below the distance 
between any two of its levels. By the uniform continuity of the homeomorphisms coming 
from the elements in $\bigcup_{j\in J}S_j$, we can then find for each $j$
an open neighbourhood of $K_j$, say $V_j$, so that $\{(S_j ,V_j )\}_{j\in J}$ is again a castle.
In view of the uniform continuity of the functions $h\in\Omega$ and of the homeomorphisms coming from $\bigcup_jS_j$,
we can also choose the neighbourhoods $V_j$ so that the
image of each tower under any of the functions of $\Omega$ has diameter less than $\delta$. Finally, we observe that
\begin{align*}
\mu \bigg( \bigsqcup_{j\in J} S_jV_j \bigg)
\geq q\sum_{j\in J} \mu (K_j ) 
&\geq q\sum_{j\in J} \bigg(\mu (E_j ) - \frac{\delta}{2q|J|} \bigg) \\
&= \mu \bigg(\bigsqcup_{j\in J}S_jE_j \bigg) - \frac{\delta}{2} 
\geq 1-\delta .  \qedhere
\end{align*}
\end{proof}

The next lemma provides us with a uniform incrementation mechanism that later in
Lemma~\ref{L-function} will be leveraged into a global condition via a maximality argument, much along the lines 
of what was done in the C$^*$-algebraic context in Section~4 of \cite{TomWhiWin15}.

\begin{lemma}\label{L-scale}
Let $G\curvearrowright X$ be a minimal action of a countably infinite discrete group
on an infinite compact metric space. Suppose that $M_G (X)$ is a Bauer simplex and that $M_G^\erg (X)$
has finite covering dimension. Then there exists an $\alpha > 0$ such that for every finite set
$\Omega\subseteq C(X,[0,1])$ and $\eta > 0$ there are $d_1 , d_2 \in C(X,[0,1])$ satisfying $d_1 d_2 = 0$ 
and $\mu (d_i h) \geq \alpha \mu (h) - \eta$ for every $i=1,2$, $h\in\Omega$, and $\mu\in M_G (X)$.
\end{lemma}

\begin{proof}
Set $\delta = \eta /4$.
Write $m$ for the covering dimension of $M_G^\erg (X)$. 
Let $\Omega$ be a finite subset of $C(X,[0,1])$. 

By Lemma~\ref{L-qshaped castles for individual measures}, given $\mu\in M_G^\mathrm{erg}(X)$
there is an open castle $\sC_\mu :=\{(S_{\mu ,j},U_{\mu ,j})\}_{j\in J_\mu}$ such that
\begin{enumerate}
\item[(a)] each shape $S_{\mu ,j}$ has cardinality $2(m+1)$,

\item[(b)] the diameters of the sets $h(S_{\mu ,j}U_{\mu,j} )\subseteq\mathbb{R}$ for $h\in\Omega$ 
are all at most $\delta/(2(m+1))$, and

\item[(c)] $\mu (\bigsqcup_{j\in J_\mu}S_{\mu ,j}U_{\mu ,j} )>1-\delta^2 /4$.
\end{enumerate}
We may assume, by reparametrizing the towers if necessary,
that all of the shapes $S_{i,j}$ contain the identity element, a fact that will be used for a couple of estimates later in the proof.
In view of (c) we can find open sets $V_{\mu ,j}\subseteq \overline{V_{\mu ,j}}\subseteq U_{\mu ,j}$ such that 
the set $C_\mu :=\bigsqcup_{j\in J_\mu}S_{\mu ,j}V_{\mu ,j}$ satisfies 
$\mu(C_\mu )>1-\delta^2 /2$.

Consider for each $\mu\in M_G^\erg (X)$ the weak$^*$ open neighbourhood 
\[
W_\mu:=\{\nu\in M_G^\erg (X): \nu(C_\mu )>1-\delta^2 /2 \} 
\]
of $\mu$.
By compactness the cover $\{W_\mu\}_{\mu\in M_G^\mathrm{erg}(X)}$ has 
a finite subcover $\{W_{\mu_1},\dots,W_{\mu_L}\}$. By the $m$-dimensionality of $M_G^{\mathrm{erg}}(X)$, 
we can refine this finite subcover to an $(m+1)$-colourable 
open cover, i.e., we can find an open cover $\{W_{i,k} : k=1,\dots,k_i,\, i=0,\dots,m\}$ so that for each $i=0,\dots,m$ the collection $\{\overline{W_{i,k}}\}_{k=1}^{k_i}$ is disjoint and 
for each $i=0,\dots,m$ and $k=1,\dots,k_i$ there is an $l(i,k)\in\{1,\dots,L\}$ such that 
$W_{i,k}\subseteq W_{\mu_{l(i,k)}}$.

Let $i\in\{0,\dots,m\}$. Set $M=\max_{k=1,\dots,k_i} |J_{\mu_{l(i,k)}}|$.
Since the sets $\overline{W_{i,k}}$ for $k=1,\dots ,k_i$ are closed and pairwise disjoint,
Lemma~\ref{L-mutual singularity} allows us to find pairwise disjoint closed sets 
$A_1,\dots,A_{k_i}\subseteq X$ such that $\mu(A_k)>1-\delta^2 /(8(m+1)^2M^2)$ 
for all $\mu\in \overline{W_{i,k}}$ and $k=1,\dots,k_i$. 
By the regularity of $X$ we can find open sets $O_k$ such that 
$A_k\subseteq O_k$ 
and the collection $\{\overline{O_k}\}_{k=1}^{k_i}$ is disjoint. 
For $k=1,\dots,k_i$ set
\[
B_{k}=\bigcap\big\{s^{-1}O_k: s\in\textstyle\bigcup_{j\in J_{\mu_{l(i,k)}}}S_{\mu_{l(i,k)},j}\big\}
\]
and
\[
\
\tilde{A}_k=\bigcap\big\{s^{-1}A_k: s\in\textstyle\bigcup_{j\in J_{\mu_{l(i,k)}}}S_{\mu_{l(i,k)},j}\big\}.
\]
Notice that for $\mu\in W_{i,k}$ we have
\begin{align}\label{eq:small compl}
\mu(X\setminus \tilde{A}_k)
&\le |J_{\mu_{l(i,k)}}|\cdot 2(m+1)\mu(X\setminus A_k)\\
&<\frac{\delta^2}{4(m+1)M} . \notag
\end{align}
Set
\[
\sC=\big\{\big(S_{\mu_{l(i,k)},j}, U_{\mu_{l(i,k)},j}\cap B_k\big): 
k=1,\dots,k_i,\, j\in J_{\mu_{l(i,k)}}\big\} ,
\]
which is an open castle whose shapes all have cardinality $2(m+1)$.
Note also, by condition (b) above, that for every tower $(T,V)$ of $\sC$
and every $h\in\Omega$ the image $h(TV)$ has diameter 
at most $\delta/(2(m+1))$.

For every $k=1,\dots,k_i$ and $j\in J_{\mu_{l(i,k)}}$ 
we use Urysohn's lemma to find a function $f_{k,j}\in C(X,[0,1])$ such that 
$f_{k,j}=1$ on $\overline{V_{\mu_{l(i,k)},j}}\cap\tilde{A}_k$ and $f_{k,j}$ is supported 
in $U_{\mu_{l(i,k)},j}\cap B_k$.
If $\mu$ is a measure in $\bigsqcup_{k=1}^{k_i}W_{i,k}$, say $\mu\in W_{i,\bar{k}}$, 
then we have 
\begin{align*}
\lefteqn{\mu\bigg(\sum_{k=1}^{k_i}\, \sum_{j\in J_{\mu_{l(i,k)}}}\sum_{s\in S_{\mu_{l(i,k)},j}}sf_{k,j}\bigg)}\hspace*{15mm}\\ 
&\geq \mu\bigg(\sum_{j\in J_{\mu_{l(i,\bar{k})}}}\sum_{s\in S_{\mu_{l(i,\bar{k})},j}}sf_{\bar{k},j}\bigg) \\
&\geq \sum_{j\in J_{\mu_{l(i,\bar{k})}}}\sum_{s\in S_{\mu_{l(i,\bar{k})},j}}
\mu(s(V_{\mu_{l(i,\bar{k})},j}\cap\tilde{A}_{\bar{k}})) \\
&= 2(m+1)\cdot\sum_{j\in J_{\mu_{l(i,\bar{k})}}}\mu(V_{\mu_{l(i,\bar{k})},j}\cap\tilde{A}_{\bar{k}}) \\
\overset{\eqref{eq:small compl}}&{\ge}2(m+1)\cdot\sum_{j\in J_{\mu_{l(i,\bar{k})}}}
\bigg(\mu(V_{\mu_{l(i,\bar{k})},j})-\frac{\delta^2}{4(m+1)M}\bigg) \\
&= \mu\bigg(\bigsqcup_{j\in J_{\mu_{l(i,\bar{k})}}}S_{\mu_{l(i,\bar{k}),j}}
V_{\mu_{l(i,\bar{k})},j}\bigg)-\frac{\delta^2|J_{\mu_{l(i,\bar{k})}}|}{2M} \\
&=\mu(C_{\mu_{l(i,\bar{k})}})-\frac{\delta^2|J_{\mu_{l(i,\bar{k})}}|}{2M} \\
&\ge\mu(C_{\mu_{l(i,\bar{k})}})-\frac{\delta^2}{2} \\
&\ge1-\delta^2 .
\end{align*}

Re-indexing everything for simplicity, what we have constructed above are
an open cover $\{ W_0 , \dots , W_m \}$ of $M_G^\erg (X)$
and, for each $i=0, \dots , m$, an open castle $\{ (S_{i,j}, V_{i,j} ) \}_{j\in J_i}$ 
and functions $f_{i,j} \in C(X,[0,1])$ for $j\in J_i$ such that 
\begin{enumerate}
\item the function $\psi_i := \sum_{j\in J_i} \sum_{s\in S_{i,j}} sf_{i,j}$
satisfies $\mu (\psi_i ) \geq 1-\delta^2$ for every $\mu\in W_i$
\end{enumerate}
and, for every $j\in J_i$,
\begin{enumerate}
\item[(ii)] the shape $S_{i,j}$ has cardinality $2(m+1)$,

\item[(iii)] the diameter of the set $h(S_{i,j} V_{i,j} )\subseteq \Rb$ is at most $\delta /(2(m+1))$ for all $h\in\Omega$,

\item[(iv)] $f_{i,j}$ is supported in $V_{i,j}$.
\end{enumerate}
Note that (i) implies that, for every $\mu\in W_i$,
\begin{align}\label{E-integral}
\mu (\{ x\in X : \psi_i (x) > 1-\delta \} ) 
\geq 1 - \delta .
\end{align}
Set $g = \sum_{i=0}^m \sum_{j\in J_i} f_{i,j}$.
For real numbers $a<b$ we write $\varphi_{a,b}$ for the continuous function on $\Rb$ which is $0$ on $(-\infty , a]$,
$1$ on $[b,\infty )$, and linear on $[a,b]$, and put
\begin{align*}
d_1 &= \varphi_{\delta , 2\delta} \circ g,\\
d_2 &= 1 - \varphi_{0 , \delta} \circ g.
\end{align*}
Then $d_1, d_2\in C(X,[0,1])$ satisfy $d_1 d_2 = 0$. 

Now let $h\in\Omega$ and suppose we are given a $\mu\in M_G^\erg (X)$.
Take an $0 \leq i\leq m$ such that $\mu\in W_i$.
Observe that
\begin{align}\label{E-approx}
\bigg| \sum_{j\in J_i} \mu (f_{i,j} h) - \frac{1}{2(m+1)} \mu (\psi_i h) \bigg| 
&= \frac{1}{2(m+1)} \bigg| \sum_{j\in J_i} \sum_{s\in S_{i,j}} 
\mu (f_{i,j} (h - s^{-1} h)) \bigg| \\
\overset{\text{(iii),(iv)}}&{\le}\frac{1}{2(m+1)}\sum_{j\in J_i}\sum_{s\in S_{i,j}}\frac{\delta}{2(m+1)}\mu(V_{i,j}) \notag \\
&\le\frac{\delta}{2(m+1)} < \delta . \notag
\end{align}
Since $d_1 + \delta 1 \geq (m+1)^{-1} g$, it follows that
\begin{align*}
\mu (d_1 h)
&\geq \frac{1}{m+1} \mu (gh) - \delta \\
&\geq \frac{1}{m+1} \mu \bigg(\bigg(\sum_{j\in J_i} f_{i,j} \bigg)h\bigg) - \delta \\
&\overset{(\ref{E-approx})}{\geq} \frac{1}{m+1} \cdot \frac{1}{2(m+1)} 
\mu (\psi_i h) - 2\delta \\
&\overset{(\ref{E-integral})}{\geq} \frac{1}{2(m+1)^2} \mu (h) - 4\delta .
\end{align*}

On the other hand, for every $\mu\in M_G^\erg (X)$ we have
\begin{align*}
\mu ((1-d_2 )h)
&= \mu ((\varphi_{0 , \delta} \circ g) h) \\
&\leq \sum_{i=0}^m \sum_{j\in J_i} \mu (1_{V_{i,j}} h) \\
\overset{\text{(iii)}}&{\leq} \sum_{i=0}^m \sum_{j\in J_i} \frac{1}{2(m+1)} 
\mu \bigg(\sum_{s\in S_{i,j}} 1_{sV_{i,j}} h\bigg) + \delta \\
&\leq \frac{m+1}{2(m+1)} \mu (h) + \delta \\
&= \frac12 \mu (h) + \delta
\end{align*}
and thus
\begin{align*}
\mu (d_2 h) = \mu (h) - \mu ((1-d_2 )h) \geq \frac12 \mu (h) - \delta.
\end{align*}
We conclude that $\mu (d_i h) \geq (2(m+1)^2)^{-1} \mu (h) - \eta$ for all $i=1,2$ 
and $\mu\in M_G^\erg (X)$,
and hence also for all $\mu\in M_G (X)$ by the Krein--Milman theorem.
\end{proof}

\begin{lemma}\label{L-function}
Let $G\curvearrowright X$ be a minimal action of a countably infinite discrete group
on an infinite compact metrizable space. Suppose that $M_G (X)$ is a Bauer simplex and that $M_G^\erg (X)$
has finite covering dimension.
Let $C$ and $D$ be disjoint closed subsets of $X$.
Then for every $\delta > 0$ there exist $h_1 ,h_2 \in C(X,[0,1])$ such that $h_1 h_2 = 0$, $h_1 =1$ on $C$, $h_2 =1$ on $D$,
and $\mu (h_1 + h_2 ) \geq 1-\delta$ for all $\mu\in M_G (X)$.
\end{lemma}

\begin{proof}
Consider the set of all $\lambda \in [0,1]$ such that 
there exist $h_1 ,h_2 \in C(X,[0,1])$ satisfying $h_1 h_2 = 0$, $h_1 =1$ on $C$, $h_2 =1$ on $D$,
and $\mu (h_1 + h_2 ) \geq \lambda$ for all $\mu\in M_G (X)$. By Urysohn's lemma this
set contains $0$ and hence is nonempty, so that it has a supremum $\beta$.
What we need to show is that $\beta = 1$.

Suppose to the contrary that $\beta < 1$. Let $\alpha > 0$ be as given by Lemma~\ref{L-scale}.
Choose an $\eta > 0$ small enough so that $\alpha (1-\beta - \eta ) > 5\eta $.
Then we can find $h_1 ,h_2 \in C(X,[0,1])$ satisfying $h_1 h_2 = 0$, $h_1 =1$ on $C$, $h_2 =1$ on $D$,
and $\mu (h_1 + h_2 ) \geq \beta - \eta$ for all $\mu\in M_G (X)$. 
With $\varphi_{a,b}$ for $a<b$ denoting as before the continuous function on $\Rb$ which is $0$ on $(-\infty , a]$,
$1$ on $[b,\infty )$, and linear on $[a,b]$, we define
\begin{align*}
f_1 &= \varphi_{2\eta,3\eta} \circ h_1 , \\
f_2 &= \varphi_{2\eta ,3\eta} \circ h_2 , \\
g_1 &= \varphi_{\eta , 2\eta} \circ h_1 - \varphi_{2\eta,3\eta} \circ h_1 , \\
g_2 &= 1 - g_1 - f_1 - f_2 .
\end{align*}
These four functions form a partition of unity in $C(X)$.
By Lemma~\ref{L-scale} there are $d_1 , d_2 \in C(X,[0,1])$ such that $d_1 d_2 = 0$ and,
for all $\mu\in M_G (X)$,
\begin{align}\label{E-alpha}
\mu (d_1 g_1) 
&\geq \alpha \mu (g_1 ) - \eta , \\
\mu (d_2 g_2 ) 
&\geq \alpha \mu (g_2 ) - \eta . \notag
\end{align}
Define
\begin{align*}
\tilde{h}_1 &= f_1 + d_1 g_1 , \\
\tilde{h}_2 &= f_2 + d_2 g_2 .
\end{align*}
Since the products $f_1 g_2$, $f_2 g_1$, $f_1 f_2$, and $d_1 d_2$ are all zero
we have $\tilde{h}_1 \tilde{h}_2 = 0$. 

Let $\mu\in M_G (X)$. If it happens that $\mu (f_1 + f_2 ) \geq \beta + \eta$ 
then we have
\begin{align*}
\mu (\tilde{h}_1 + \tilde{h}_2 ) 
\geq \mu (f_1 + f_2 ) 
\geq \beta + \eta .
\end{align*}
If on the other hand $\mu (f_1 + f_2 ) < \beta + \eta$ then
\begin{align*}
\mu (g_1 + g_2) 
= 1 - \mu (f_1 + f_2 ) 
> 1 - \beta - \eta
\end{align*}
and hence, using the inequalities $f_1 + f_2 \geq h_1 + h_2 - 2\eta$ and (\ref{E-alpha})
along with our choice of $\eta$,
\begin{align*}
\mu (\tilde{h}_1 + \tilde{h}_2 ) 
&= \mu (f_1 + f_2 ) + \mu (d_1 g_1) + \mu (d_2 g_2) \\
&\geq \mu (h_1 + h_2 ) - 2\eta + \alpha \mu (g_1 + g_2 ) - 2\eta \\
&> \beta - 5\eta + \alpha (1-\beta - \eta) \\
&> \beta
\end{align*}
Thus in either case we have $\mu (\tilde{h}_1 + \tilde{h}_2 ) > \beta$,
contradicting the definition of $\beta$.
\end{proof}

\begin{proof}[Proof of Theorem~\ref{T-main}]
We may assume that $X$ is infinite, for otherwise the SBP is obvious.
Let $C$ and $D$ be disjoint closed subsets of $X$ and let $\delta > 0$.
By the characterization of the SBP mentioned in Section~\ref{S-preliminaries},
it suffices to find an open set $U\subseteq X$ such that $C\subseteq U \subseteq X\setminus D$
and $\mu (\partial U) < \delta$ for every $\mu\in M_G (X)$.
By Lemma~\ref{L-function} there exist $h_1 ,h_2 \in C(X,[0,1])$ such that $h_1 h_2 = 0$, $h_1 =1$ on $C$, $h_2 =1$ on $D$,
and $\mu (h_1 + h_2 ) \geq 1-\delta$ for all $\mu\in M_G (X)$.
Set $U = \{ x\in X : h_1 (x) > 0 \}$. Then $h_1 = h_2 = 0$ on $\partial U$, so that for every $\mu\in M_G (X)$ we have
\begin{align*}
\mu (\partial U) \leq 1 - \mu (h_1 + h_2 ) \leq \delta .
\end{align*}
Since $C\subseteq U \subseteq X\setminus D$, we are done.
\end{proof}

The following provides some examples of actions of locally finite groups on infinite-dimensional spaces
to which Theorem~\ref{T-main} applies.

\begin{example}
Using the same kind of construction carried out in \cite{Tho94} in the C$^*$-algebraic framework of simple interval algebras,
we can build the following minimal actions $G\curvearrowright X$ of UHF-type limits of permutation groups 
with prescribed metrizable Choquet simplex as the space of invariant Borel probability measures and with $X$ of 
infinite covering dimension.  
Let $m_1 , m_2 , \dots$ be a sequence of integers greater than $1$.
For each $k$ set $n_k = m_1 m_2 \cdots m_k$ and define $G_k$ to be the permutation group $\Sym (n_k )$.

Write $Y$ for the Hilbert cube and set $X_k = Y \times \{ 0,\dots , n_k -1 \}$.
The construction will also work for other path-connected compact metrizable spaces $Y$, such 
as finite-dimensional cubes, but we are interested here in examples of the SBP beyond the setting of
finite-dimensional spaces, where it is automatic assuming the action is free \cite{LinWei00,Sza15}
(another possibility would be to let the space $Y$ depend on $k$ and tend to infinity in dimension, 
like in the construction of many simple AH algebras, including the C$^*$-algebras of Villadsen and Toms 
in \cite{Vil98} and \cite{Tom08},
whose $\Zb$-dynamical analogues do not have the SBP \cite{LinWei00,GioKer10}).
For each $k>1$ let $f_{k,0} , \dots , f_{k,m_k-1} : Y \to Y$ be continuous maps.
Then we define a continuous connecting map $\varphi_{k+1} : X_{k+1} \to X_k$ by
\[
\varphi_{k+1} ((y,an_k+b)) = (f_{k+1,a} (y),b)
\]
for $a=0,\dots , m_{k+1}-1$ and $b=0,\dots ,n_k-1$.
The group $G_k$ acts on $X_k$ by permuting the second coordinate, and these actions
induce embeddings $\gamma_k : G_k \hookrightarrow G_{k+1}$ such that
$s\varphi_{k+1} (x) = \varphi_{k+1}(\gamma_k (s) x)$ for all $x\in X_{k+1}$ and $s\in G_k$.
Write $X$ for the inverse limit of the spaces $X_k$ and $G$ for the direct limit of the groups $G_k$.
Then $G$ is a countably infinite locally finite group and we obtain a continuous action 
$G\curvearrowright X$ on a compact metrizable space.

Now suppose that for every $k$ one has $\bigcup_{a=0}^{m_{k+1}-1} f_{k+1,a} ( Y ) = Y$.
Then the connecting maps $\varphi_k$ are surjective. 
Moreover, one can check in this situation that the action will be minimal precisely when for every $k\in\Nb$ and
nonempty open set $U\subseteq Y$ there is an integer $l\geq 0$ such that for every $y\in Y$ there are 
$a_j \in \{ 0,\dots , m_k -1 \}$ for $j=0,\dots , l$ for which
\[
(f_{k,a_0} \circ f_{k+1,a_1} \circ\cdots\circ f_{k+l,a_l} )(y) \in U
\]
(compare Lemma~1.2 in \cite{Tho94} and Proposition~2.1 in \cite{DadNagNemPas92}).
The construction in Section~3 of \cite{Tho94} shows that, given any metrizable Choquet simplex $S$, 
if we choose the numbers $m_k$ to grow fast enough then one can choose functions $f_{k,i}$ satisfying
these conditions in such a way that $M_G (X)$ is affinely homeomorphic to $S$. Note that Lemma~3.8 of \cite{Tho94} works just as well with almost the same proof even if $P$ therein is replaced with the simplex of Borel probability measures on any compact metrizable path-connected space that is not a point (such as the Hilbert cube).
Note moreover that, in the proof of Lemma~3.7 in \cite{Tho94}, for each $k$ a modification of two of the initially chosen
functions $f_{k,i}$ is made so as to produce an iterated dyadic splitting of the interval
that guarantees simplicity (which translates as minimality in
our context) without changing the trace simplex in the limit. It is for the latter purpose that 
the growth condition on the numbers $m_k$ is used, via Lemma~3.5 in \cite{Tho94}. In our case we 
can still dyadically split to the same effect but now in a way that incorporates larger and larger finite sets of coordinates
in the Hilbert cube, so that we now need to modify a growing number of the $f_{k,i}$ from stage to stage. 
If moreover we set one of the $f_{k,i}$ for each $k$ to be the identity map 
then the space $X$ contains closed subsets homeomorphic to $Y$ and hence
is infinite-dimensional. These modifications can all be absorbed at the measure-theoretic level
(i.e., without changing the simplex of invariant Borel probability measures in the limit)
by requiring that the numbers $m_k$ grow even faster.

The above actions are closely related to the orbit-breaking techniques that appear in the study
of C$^*$-algebra crossed products. For example, if $Y$ is a singleton then, up to orbit equivalence,
we get a $\Zb$-odometer but with one orbit cut into two pieces where the infinite rollover occurs.
This gives rise to an AF algebra (the C$^*$-algebraic analogue of local finiteness, or of a locally finite group
acting on the Cantor set) 
inside the crossed product of the odometer, with the inclusion inducing an isomorphism of $K_0$ groups.
The crossed product itself is not AF 
(the acting group $\Zb$ yields a $K_1$ obstruction) but rather a Bunce--Deddens algebra.
See Example~VIII.6.3 and Section~VIII.7 of \cite{Dav96}.
\end{example}

\end{document}